\documentclass[12pt]{article}
\title{A proof of the Borel completeness of torsion free abelian groups}
\author{ Michael C.\
Laskowski  
and Douglas S.\ Ulrich
 \thanks{Both authors partially supported
by NSF grant DMS-1855789.}
\\
Department of Mathematics\\University of Maryland
}
\usepackage{amssymb}
\usepackage{enumitem}
\usepackage{mathtools}
\usepackage{geometry,amssymb,amsthm,amsmath,
graphics}
\usepackage{times}
\usepackage{amsfonts}
\usepackage{amssymb}
\usepackage{amscd}
\usepackage{url}

\def\includeE#1{{\lhook\kern-3.5pt\joinrel\smash{
    \mathop{\longrightarrow}\limits^{#1}}}}

\def\efor/{Example~\ref{E4}}

\def\BL/{Baldwin--Lachlan}
\def\Bu/{Buechler}
\def\Hr/{Hrushovski}
\def\lm/{locally modular}
\def\wm/{weakly minimal}
\def\nm/{non--modular}
\def\ss/{superstable}
\def\ud/{unidimensional}
\def\sm/{strongly minimal}

\def\hbar{\overline{h}}

\def\tr/{trivial}
\def\nt/{non--trivial}
\def\st/{strong type}

\def\phi{\varphi}

\def\F{{\cal F}}
\def\FF{{\bf F}}

\def\Z{{\mathbb Z}}

\def\Fa0{{\FF^a_{\aleph_0}}}

\def\endproof{\medskip}
\def\<{\langle}
\def\>{\rangle}

\def\FF{{\mathbb F}}
\newcommand\myrestriction{\mathord\restriction}
\def\mr#1{\myrestriction_{#1}}

\newtheorem{Theorem}{Theorem}[section]
\newtheorem{Proposition}[Theorem]{Proposition}
\newtheorem{Definition}[Theorem]{Definition}

\newtheorem{Remark}[Theorem]{Remark}

\begin{document}

\date{\today}

\maketitle

\section{Introduction}

In February 2021, Paolini and Shelah \cite{PS} announced a proof that the theory TFAB of torsion free abelian groups is Borel complete.  Since then, it has been observed and
corroborated by them that there is a gap in both the February 2021 and June 2021 arXiv submissions.  We understand that Paolini and Shelah are actively working on a revised version of their proof.
Although their work precedes ours,  it seems like the proof of this fact presented in this article is sufficiently different to be of interest in its own right.  
 
 In some sense the proof presented here, in contrast to the approach of Paolini-Shelah, is a continuation of an idea of Shelah and the second author.  
 Although differently named there, in \cite{ShU} Shelah and Ulrich explore the class of {\em tagged abelian groups}\footnote{This notion, with still a different name, appears in work of G\"obel-Shelah~\cite{GS}
 and is  related to Hjorth's notion of an ``eplag'' in \cite{Hjorth}.}
 and prove that this class of structures is Borel equivalent to TFAB.
% described a novel class of expansions of abelian groups.
% The salient definition from \cite{ShU} is of a tagged abelian group,\footnote{Shelah and Ulrich say this notion is a
% generalization of Hjorth's notion of ``eplag.''} although the notion was differently named there.

\begin{Definition} \label{tagged} {\em  Let $L=\{+,0\}\cup\{U_n:n\in\omega\}$.  A {\em tagged abelian group} is an $L$-structure $M$ so that $(M,+,0)$ is an abelian group and
every $U_n^M$ is a subgroup of $M$.
}
\end{Definition}

Much of the power of this notion is its malleability.  We record one manifestation of this flexibility for later reference.

\begin{Remark}  \label{shuffle}  {\em At its heart, a tagged abelian group is simply an abelian group $(M,+,0)$, together with countably many subgroups named by unary predicates.
Thus, if an $L$-structure $M$ is a tagged abelian group, then any expansion $M^*$ of $M$ formed by naming one (or even countably many) new subgroups can  be construed 
as a tagged abelian group by reassinging the names of the predicate symbols.
%In particular, if an $L$-structure is a tagged abelian group, then (up to a reshuffling of the predicate symbols) so is any expansion by a unary predicate, provided the new predicate is
%interpreted as a subgroup. 
}
\end{Remark}
%\begin{Remark}  \label{shuffle}  {\em At its heart, a tagged abelian group is simply an abelian group $(M,+,0)$, together with countably many subgroups named by unary predicates.
%In particular, if an $L$-structure is a tagged abelian group, then (up to a reshuffling of the predicate symbols) so is any expansion by a unary predicate, provided the new predicate is
%interpreted as a subgroup. 
%}
%\end{Remark}

A consequence of  Theorem~13 of \cite{ShU} is that there is a Borel reduction from the class of (countable) tagged abelian groups to TFAB, the class of torsion-free abelian groups.
A road map of the proof of this consequence is given in Section~\ref{old}.   In light of that result, it suffices to show that tagged abelian groups are Borel complete, which we prove in Section~\ref{new}.

\section{Summarizing (tagged abelian groups)$\le_B$ TFAB}   \label{old}

As noted above, the following appears as Theorem~13 of \cite{ShU}.

\begin{Theorem}[Shelah-Ulrich]  The classes of (countable) Tagged Abelian Groups, Abelian Groups, and Torsion-free Abelian Groups are Borel equivalent.
In particular, if Tagged Abelian Groups is Borel complete, then so is the class of Torsion-free Abelian Groups.
\end{Theorem}

\begin{proof}  As it is what we need, we 
 outline the steps in showing that tagged abelian groups are Borel reducible to
TFAB.
To  ease readability, we record three classes of objects that were defined in \cite{ShU}.  Throughout the whole of this argument, we can take
 ${\cal J}$, the set of distinguished homomorphisms, to be empty.   
\begin{itemize}
\item  $\Omega_{\omega,\emptyset}^-$ there is precisely the class of countable, tagged abelian groups.
\item  $\Omega_{\omega,\emptyset}$ is the class of tagged abelian groups with universe $\bigoplus_\omega \Z$.
\item  $\Omega_{\omega,\emptyset}^p$ is the class of tagged abelian groups with universe $\bigoplus_\omega \Z$, where we further require that
each of the distinguished subgroups $U_n$ is a pure subgroup of $\bigoplus_\omega \Z$.
\end{itemize}

%Now, Theorem~\ref{tagBC} above states that $\Omega_{\omega,\emptyset}^-$ is Borel complete.
 Theorem~22 of  \cite{ShU} states  that $\Omega_{\omega,\emptyset}^-\le_B \Omega_{\omega,\emptyset}$.
Continuing, from Lemma~17 there, $\Omega_{\omega,\emptyset}\le_B \Omega_{\omega,\emptyset}^p$,
and $\Omega_{\omega,\emptyset}^p\le_B TFAB$ by
 Lemma~19 there.  [Remark:  Lemma~19 can also be established by more ad hoc methods that do not involve tensor products.  An alternate argument is available upon request.]
 \qed
\end{proof}

\section{Tagged abelian groups are Borel complete.}  \label{new}

\begin{Theorem} \label{tagBC} 
 The class of countable, tagged abelian groups is Borel complete.  In fact, we show that the class of tagged abelian groups whose universe is $\bigoplus_\omega\FF_2$ is Borel complete.
\end{Theorem}

  Our strategy for proving Theorem~\ref{tagBC}  to construct a single,
 countable, tagged abelian group $N$ satisfying the conditions of Proposition~\ref{lift}.
Then, using the fact mentioned in Remark~\ref{shuffle}, we will give a Borel mapping between from the space of countable graphs $G=(\omega,R)$ 
to expansions $N(R)=(N,U_*(R))$  of $N$ by a subgroup preserving isomorphism in both directions.  

%Thus, it suffices to prove Theorem~\ref{tagBC}.  
In everything that follows, 
 we will be looking at tagged abelian groups where the universe  $M$ is an $\FF_2$-vector space.  Note that in such cases, each $U_n^M$ is automatically an
$\FF_2$-subspace of $M$.
Call a subset $Y\subseteq M^k$  {\em invariant} if every $\sigma\in Aut(M)$ fixes $Y$ setwise.  Clearly, if $Y$ is 0-definable by a formula in $L_{\omega_1,\omega}$,
then $Y$ is invariant.  As a simple example, for any tagged abelian group $M$, the set $X^M:=M\setminus\bigcup\{U_n^M:n\in\omega\}$ is invariant.

\smallskip

Throughout this note, for any set $Z$, $Sym(Z)$ denotes the set of permutations of $Z$.  We isolate a central notion.

\begin{Definition}  {\em  
Suppose $M$ is any structure, $X\subseteq M$ any subset, and $E\subseteq X^2$ any equivalence relation on $X$.  
We say a particular permutation
{\em $h\in Sym(X/E)$ lifts to an automorphism of $M$} if there is $\sigma\in Aut(M)$ fixing $X$ setwise and satisfying $\sigma(x)/E=h(x/E)$ for every $x\in X$.
%We say {\em $(X,E)$ lifts to $Aut(M)$} if every $h\in Sym(M)$ lifts to an automorphism of $M$.
}
\end{Definition}

\begin{Proposition}  \label{herring} There is a tagged abelian group $M$ with universe $\bigoplus_\omega\FF_2$ such that:
\begin{enumerate}
\item  $X:=M\setminus\bigcup\{U_n^M:n\in\omega\}$ is a basis for the $\FF_2$-vector space $(M,+,0)$;
\item  There is an equivalence relation $E\subseteq X^2$ with infinitely many classes such that every $h\in Sym(X/E)$ lifts to some $\sigma\in Aut(M)$.
\end{enumerate}
\end{Proposition}

\begin{proof}  We prove Proposition~\ref{herring} in two parts.  For the first, we exhibit a class $K_0$ of (finite) tagged abelian groups $A$, each with a distinguished subset $X^{A}$ satisfying disjoint amalgamation.
The $L$-reduct of the $K_0$-Fraisse limit would satisfy (1), but to define $E$ satisfying (2), we   modify a technique from \cite{Herrings}.

%We describe a class  $K$ of two-sorted structures $A=A^R\sqcup A^S$ that also satisfies disjoint amalgamation for which
%the $R$-part of an element of $K$ can naturally be construed as an element of $K_0$, while the $S$-part, and more importantly the function $p$ between the sorts approximates the equivalence relation $E$.
%We will recover $M$, $X$,  and an equivalence relation $E$ on $X$ from the $K$-Fraisse limit.

%We begin by defining a desirable class of finite tagged abelian groups and showing they satisfy disjoint amalgamation.

To begin, let $L_X:=L\cup\{X\}$, where an additional $X$ is a unary predicate, but we do not interpret it as a subgroup.
Let 
$K_0$ be the class of (finite)  $L_X$ structures $A$ such that

\begin{itemize}
\item  The $L$-reduct of $A$ is a tagged abelian group with $(A,+,0)\cong \FF_2^k$ for some integer $k$;
\item  $U_n^A=\{0\}$ for all but finitely many $n\in\omega$; 
\item  $X^A:=A\setminus\bigcup\{U_n^A:n\in\omega\}$; and
\item  $X^A$ is a basis for the $\FF_2$-vector space $(A,+,0)$.
\end{itemize}
To verify that $K_0$ satisfies disjoint amalgamation,   suppose $A,B,C\in K_0$ with $A$ an $L$-substructure of both $B$ and $C$ with $B\cap C=A$.
Let $(D,+,0)$ be any  $\FF_2$-vector space with basis $X^D:=X^B\cup X^C$ for which $B$ and $C$ are subspaces.
%consisting of all formal sums $b+c$ with $b\in B$, $c\in C$ and put $X^D:=X^B\cup X^C$.
%Since $X^A=X^B\cap X^C$,  modularity implies that $X^D:=X^B\cup X^C$ is a basis for $(D,+,0)$.
To define the predicates $U_n^D$, we split into cases.  Call $n\in\omega$ {\em active} if either $U_n^B$ or $U_n^C\neq \{0\}$.
For each active $n$, put $U_n^D:=\<U_n^B\cup U_n^C\>$, but note that cofinitely many $n$ are inactive. 
 
%Now $X^A=X^B\cap X^C$, hence by modularity $X^D:=X^B\cup X^C$ is a basis for $(D,+,0)$.  
Clearly $X^D\subseteq D\setminus\{U_n^D:n$ active$\}$, but equality need not hold.
For each of the (finitely many) $a$ in the difference, attach an inactive index $n(a)$ and put $U_{n(a)}^D:=\{0,a\}$.  Finally, complete the description of $D$ by positing $U_n^D:=\{0\}$
for the cofinite set of $n$'s that are neither active, nor of the form $n(a)$ for any $a$ in the difference.  It is easily checked that $D\in K_0$.

\medskip
Let $M$ denote the $L$-reduct of the $K_0$-Fraisse limit.  It is easily checked that
$M$ is a tagged abelian group, $X^M$ is a basis, $M\setminus \{U_n^M:n\in\omega\}=X^M$ and because $K_0$ has elements $A$ with $X^A$ arbitrarily large (but finite), $X^M$ is infinite.   

\bigskip

%Next, we argue as in \cite{Herrings} to construct an equivalence relation $E$ on 
%\footnote{The method here is akin to 1.10 of \cite{Herrings}, with the primary difference being that the domain of $p$ is only $X^A$ as opposed to the whole of $A^R$.} 

For the second part, we add a binary relation symbol $E$ to the language, i.e., $L_{XE}:=L\cup\{X,E\}$
and define a class $K$ of finite $L_{XE}$-structures $A=(A_0,E^A)$ satisfying
\begin{itemize}
\item  $A_0\in K_0$; and
\item  $E^{A}\subseteq X_{A_{0}}^2$ is an equivalence relation.
\end{itemize}

We first note that $K$ satisfies disjoint amalgamation in a very strong way.  The following is immediate from the definition of the class $K$.

\begin{quotation}

\noindent   For all $A,B,C\in K$ with $A\subseteq B$, $A\subseteq C$, if $D_0\in K_0$ is any amalgam of $B_0$ and $C_0$ with $X^{D_0}=X^{B_0}\cup X^{C_0}$, then for {\bf any} equivalence
relation $E^*$ on $X^{D_0}$ extending $E^B\cup E^C$, $D:=(D_0,E^*)\in K$ and $B\subseteq D$, $C\subseteq D)$.
\end{quotation}

Let $M^*$ be the Fraisse limit with respect to the class $K$.  To ease notation, let $X:=X^{M^*}$ and $E:=E^{M^*}$.  
Then as above, the $L$-reduct $M$ of $M^*$ is a tagged abelian group with $(M,+,0)$ isomorphic to $\bigoplus_\omega \FF_2$ with $X$ a basis and $X=M\setminus\bigcup\{U_n^{M^*}:n\in\omega\}$.
Additionally $E$ is an equivalence relation on $X$.  

\medskip
\noindent{\bf Claim 1.}  $X/E$ is infinite.
\medskip

\begin{proof}  For each $n$, choose $A_0\in K_0$ with $|X^{A_0}|\ge n$ and let $E^{{\rm tr}}=\{(x,x):x\in X^{A_0}$ denote the trivial equivalence relation.
Then $A:=(A_0,E^{{\rm tr}})\in K$ and $|X^{A_0}/E^{{\rm tr}}|=|X^{A_0}|\ge n$.  As the $L_{XE}$-structure $A$ embeds into $M^*$, $|X/E|\ge n$.
As $n$ is arbitrary, $X/E$ is infinite.
\end{proof}

\medskip
\noindent{\bf Claim 2.}  Every $h\in Sym(X/E)$ lifts to an $L_{XE}$-automorphism of $M^*$.
\medskip

\begin{proof}  Fix $h\in Sym(X/E)$ and let $\F_h$ denote the set of $L_{XE}$-isomorphisms $f:A\rightarrow B$, where $A,B\in K$ and $A,B\subseteq M^*$ satisfying
\begin{itemize}
\item  For all $x\in X^A$, $f(x)/E=h(x/E)$.
\end{itemize}

It suffices to prove that $\F_h$ is a back-and-forth system.  That $\F_h\neq\emptyset$ is easy (take $A=B$ to have universe $\{0\}$), so fix any
$f:A\rightarrow B$ and any $A'\in K$ with $A\subseteq A'\subseteq M^*$.  It suffices to show there is an extension $f':A'\rightarrow B'$ of $f$ with $f'\in \F_h$.
As a first step toward building $f'$, choose an abstract $L_{XE}$-isomorphism $g:A'\rightarrow C$ extending $f$ with $C\cap M^*=B$.
Next, to make the $E$-classes align, choose a finite set $Y$, $X^B\subseteq Y\subseteq X$ such that for every $x\in X^{A'}$, there is $y\in Y$ with $y\in h(x/E)$.
Let $F:=(F_0,E^F)\subseteq M^*$ be the smallest $L_{XE}$-substructure containing $Y$.  In particular, $X^F=Y$ and $E^F$ is simply the equivalence relation on $Y$ induced by $E^{M^*}$.
The structure $F$ will not appear in the image of the $f'\in\F_h$ we build.  Rather, it is used to maintain control of the $E$-classes of $f'(x)$ for $x\in X^{A'}$.

Now $F\in K$ and $B\subseteq F$.  Since $M^*\cap C=B$, we have $F\cap C=B$.  We construct $D=(D_0,E^D)\in K$ with $F\subseteq D$ and $C\subseteq D$ as follows.
As $B_0,F_0,C_0\in K_0$, by the disjoint amalgamation for $K_0$ described above, let $D_0\in K_0$ be such that $F_0\subseteq D_0$, $C_0\subseteq D_0$ and $X^{D_0}=X^{F_0}\cup X^{C_0}$;
and let $E^D$ be the (unique!) equivalence relation on $X^{D_0}$ extending $E^{F}\cup E_C$ and (recalling $X^{F_0}=Y$)
\begin{itemize}  
\item  For all $x\in X^{A'}$, $y\in Y$ $E^D(g(x),y)\Leftrightarrow h(x)/E=y/E$.
\end{itemize}
As  $F\subseteq D$ and $M^*$ is $K$-homogeneous, choose an $L_{XE}$-embedding $k:D\rightarrow M^*$ fixing $F$ pointwise.

Finally, let $f':A'\rightarrow B'$ be the composition $k\circ g$, where $B'$ is the image of $k\mr{C}$.
Clearly, $f'$ is an $L_{XE}$-isomorphism extending $f$ and $B'\in K$, $B'\subseteq M^*$.  It remains to check that $f'(x)/E=h(x/E)$ for every $x\in X^{A'}$.
To see this, fix any $x\in X^{A'}$.  By our choice of $Y$, choose $y\in Y=X^F$ with $y/E=h(x/E)$.
By the displayed equation, $E^D(g(x),y)$ holds.  Since $k$ is an embedding fixing $F$ pointwise, this implies $E(k\circ g(x),y)$ (recalling $E=E^{M^*}$).
Thus, $f'(x)/E=y/E=h(x/E)$, completing the proof of Claim 2 and hence of the Proposition.
\qed
\end{proof}
\end{proof}

Using the data described in Proposition~\ref{herring}  we construct our `engine' $N$ for coding graphs.

\begin{Proposition}  \label{lift} There is a tagged abelian group $N$ with universe $\bigoplus_\omega\FF_2\oplus\bigoplus_\omega \FF_2$ and $Aut(N)$-invariant sets $X_0,X_1\subseteq N$
and invariant $E_0\subseteq X_0^2$, $E_1\subseteq X_1^2$ such that 
\begin{enumerate}
\item  $X_0\cup X_1$ is linearly independent;
\item  each $E_i$ is an equivalence relation on $X_i$ with infinitely many classes;
\item  every $h\in Sym(X_0/E_0)$ lifts to an automorphism $\sigma\in Aut(N)$; and
\item  There exists an $Aut(N)$-invariant bijection 
$k^*:[X_0]^2\rightarrow X_1/E_1$
i.e., for all $\{x,y\}\in [X_0]^2$, for all $z\in X_1$, and for all $\sigma\in Aut(N)$, $k^*(x,y)=z/E_1$ if and only if $k^*(\sigma(x),\sigma(y))=\sigma(z)/E_1$.
%Moreover, we may additionally assume that $U_n^N=\{0\}$.
\end{enumerate}
\end{Proposition}

\begin{proof} As noted in Remark~\ref{shuffle},  when describing a tagged abelian group, the ordering of the $U_n$'s is irrelevant.  Hence,  
for clarity we can use other names for unary predicate symbols (so long as there
are only countably many predicates named).  Accordingly, our $N$ should be thought of as a direct sum $M\oplus M$, with $M$ from Proposition~\ref{herring}, endowed with two additional predicate symbols
$W_0,W_1$ that will be defined momentarily.   More formally, take $N$ to have universe $\bigoplus_\omega\FF_2\oplus \bigoplus_\omega\FF_2$ and (not as part of the language) for $i\in\{0,1\}$,
 let $\pi_i:N\rightarrow M$ be the natural projection maps.  Let $V_0,V_1$ be unary predicate symbols from our language interpreted as $V_0^N=M\times\{0\}$ and $V_1^N=\{0\}\times M$.
 As well, for $i\in\{0,1\}$ we have infinitely many unary predicates $V_{i,n}$ interpreted as subspaces of $V_i$ corresponding to the subspaces $U_n\subseteq M$.  
 That is, up to reindexing of the predicates, each of $(V_0,V_{0,n})_{n\in\omega}$ and $(V_1,V_{1,n})_{n\in\omega}$ are isomorphic as tagged abelian groups to $(M,U_n)_{n\in\omega}$.
 For $i\in\{0,1\}$, let $X_i:=V_i\setminus \bigcup\{V_{i,n}:n\in\omega\}$ and let $E_i\subseteq X_i$ be the corresponding equivalence relation given by Proposition~\ref{herring}.
 It follows from the properties of $(M,U_n)_{n\in\omega}$ that for $i=0,1$, $X_i$ is a basis for $V_i$ and is invariant under any $\sigma\in Aut(N)$.  Moreover,
 any permutation $h_i\in Sym(X_i/E_i)$ lifts to an automorphism of $(V_i,V_{i,n})_{n\in\omega}$.
 
 Up to here, all we have defined is a `doubling' of $M$.  We now give interpretations to the two new predicates $W_0,W_1$ to make $E_0,E_1$ invariant and to define $k^*$ as in the statement.
 First, let $S_0:=\{x+y:\{x,y\}\in [X_0]^2\}\subseteq V_0$.  Since $X_0\subseteq V_0$ is linearly independent, the map $\{x,y\}\mapsto x+y$ is a bijection $j:[X_0]^2\rightarrow S_0$.
 Next, as $S_0$ and $X_1/E_1$ are both countably infinite sets, fix a bijection $k:S_0\rightarrow X_1/E_1$ and let $k^*:[X_0]^2\rightarrow X_1/E_1$
 be the bijection $k\circ j$.  Note that the sets $\{k(s):s\in S_0\}$ describe a partition of $X_1$.
 Let 
 $$T:=\{s+z:s\in S_0, z\in X_1, z/E_1=k(s)\}$$
 and let $W_0^N:=\<T\>$, the $\FF_2$-subspace spanned by $T$.
 Additionally, let
 %$$Q:=\bigcup\{k(s):s=x+y\ \hbox{for some $\{x,y\}\in [X_0]^2$ with $E_0(x,y)\}$}$$
 $$Q:=\{z\in X_1: z/E_1\in k(x+y)\  \hbox{for some $\{x,y\}\in [X_0]^2$ with $E_0(x,y)\}$}$$
 and let $W_1^N:=\<Q\>$, which is an $\FF_2$-subspace of $V_1$.
 
 \medskip
 \noindent{\bf Claim 1.}  For every $z\in X_1$, there is exactly one $w\in W_0^N$ such that $\pi_1(w)=z$.  In fact, any such $w$ is an element of  $T$.
 \medskip
 
 \begin{proof}  First, note that  since $\{k(s):s\in S_0\}$ partitions $X_1$,
 for every $z\in X_1$ there is precisely one $s\in S_0$ so that $s+z\in T$.  
 Now fix any $z\in X_1$.  Choosing $s\in S_0$ as in the previous sentence gives $s+z\in T\subseteq W_0$.
 We show that $s+z$ is the only $w\in W_0$ with $\pi_1(w)=z$.
 To see this, choose any $w=\sum_{i<n} s_i+z_i$ with each $s_i+z_i\in T$ and $\pi_1(w)=z$.  We may assume that $\sum_{i<n} s_i+z_i$ is `reduced' in the sense that
 $(s_i+z_i)+ (s_j+z_j)\neq (0,0)$ for all distinct $i,j<n$.  Under this assumption, we claim that $n=1$, i.e., $w=s+z$ as described above.
 To verify that $n=1$, 
 note that $\{z_i:i<n\}$ are distinct.  Indeed, if $z_i=z_j$, then by the uniqueness described in the first sentence of the proof
 we would have $s_i+z_i=s_j+z_j$.  But then, as $N$ has exponent 2, $(s_i+z_i)+(s_j+z_j)=(0,0)$ contradicting that $\sum_{i<n} s_i+z_i$ is reduced.
 As $\{s_i:i<n\}\subseteq V_0$, while $\{z_i:i<n\}\subseteq V_1$, $N$ being a direct sum and $\sum_{i<n} s_i+z_i=s+z$  implies
 $\sum_{i<n} z_i=z$.  Coupling this with the  fact that
 $\{z_i:i<n\}$ are distinct, the linear independence of $X_1$ implies that
 $n=1$ and $z_0=z$.
 \qed
 \end{proof}

% To get that there is precisely one $w\in W_0$ such that
% 
% 
% Fix $z\in X_1$.  As $\{k(s):s\in S_0\}$ partition $X_1$, choose the unique $s\in S_0$ so that $z/E_1=k(s)$.  Then clearly,    $(s,z)\in T$ and no other $(s',z)\in T$.
% We argue that $(s,z)$ is the only element of $W_0^N$ with $\pi_1(w)=z$.  To see this, suppose $w=\sum_{i<n} (s_i,z_i)$ with each $(s_i,z_i)\in T$.
% Furthermore, assume the sum is minimal in the sense that for distinct $i,j<n$, $(s_i,z_i)+(s_j,z_j)\neq (0,0)$ and we will show $n=1$.
% 
% 
%  As $N$ has exponent two, we
% may assume the elements $(s_i,z_i)$ are without repetition.  From the above, this implies the set $\{z_i:i<n\}\subseteq X_1$ are distinct.  As $X_1$ is linearly independent, $z\in X_1$
% and $z=\sum_{i<n} z_i$ imply that $n=1$, i.e., $w=(s,z)\in T$ as claimed.
% \endproof
 
 \medskip
 With this in hand, we now verify that every $h\in Sym(X_0/E_0)$ lifts to an automorphism of $N$.  Fix such an $h$ and since $(M,U_n)_{n\in\omega}$ lifts,
 choose $\sigma_0\in Aut(V_0,V_{0,n})_{n\in\omega}$ such that for all $x\in V_0$ $\sigma_1(x)/E_0=h(x/E_0)$.
 Also, as $\sigma_0$ permutes $X_0$, it permutes $S_0$.  Since $k:S_0\rightarrow X_1/E_1$ is a bijection, we get an induced $h_1\in Sym(X_1/E_1)$ defined as:
 
 $$h_1(k(s))=k(\sigma_0(s))$$
 Using the lifting property of $h_1\in Sym(X_1/E_1)$, choose $\sigma_1\in Aut(V_1,V_{1,n})_{n\in\omega}$ satisfying: 
 
 \begin{quotation}  \noindent For all $z\in X_1$, if $z/E_1=k(s)$, then
 $\sigma_1(z)/E_1=h_1(z/E_1)=k(\sigma_0(s)).$
 \end{quotation}
 Putting these together, as $N=V_0\oplus V_1$, let $\sigma:N\rightarrow N$ be defined as $\sigma(x+z)=\sigma_0(x)+\sigma_1(z)$.
 It is evident that $\sigma$ is a bijective homomorphism fixing $V_i$ and $V_{i,n}$ setwise for each $i\in\{0,1\}$ and all $n\in\omega$.  
 
 We show that this $\sigma$ also fixes $W_0^N$ and $W_1^N$ setwise.  Since $W_0^N=\<T\>$ and $W_1^N=\<Q\>$, it suffices to show that $\sigma$ fixes both $T$ and $Q$ setwise.
 To verify these, choose $(s,z)\in T$.  Since $\sigma(s,z)=(\sigma_0(s),\sigma_1(z))$, the equation $\sigma_1(z)/E_1=k(\sigma_0(s))$ implies that $\sigma(s,z)\in T$, and the reverse direction is symmetric.
 As for $Q$, choose $z\in Q$ and choose $\{x,y\}\in [X_0]^2\cap E_0$ such that $z/E_1=k(x+y)$.  As $\sigma_0$ is a lifting, it preserves $E_0$, hence $E_0(\sigma_0(x),\sigma_0(y))$.
 By our choice of $\sigma_1$. we have $\sigma_1(z)/E_1=k(\sigma_0(x)+\sigma_0(y))$.  Thus, $\{\sigma_0(x),\sigma_0(y)\}\in [X_0]^2\cap E_0$ witness that $\sigma_1(z)\in Q$.  Again, the reverse direction is
 symmetric.    Thus, every $h\in Sym(X_0/E_0)$ lifts to an automorphism of $N$.
 
 The remaining clauses of  Proposition~\ref{lift} are collected in the following claim.
 
 \medskip
 \noindent{\bf Claim 2.}  Suppose $\sigma\in Aut(N)$,  $\{x,y\}\in [X_0]^2$ and $z\in X_1$ with $z/E_1=k^*(\{x,y\})$.  Then:
 \begin{enumerate}[label=(\roman*)]
 \item  $\sigma(z)/E_1=k^*(\{\sigma(x),\sigma(y)\})$;
 \item  $\sigma$ preserves $E_1\subseteq X_1^2$ setwise;
 \item  $W_1^N\cap X_1=Q$, hence $\sigma$ fixes $Q$ setwise; and
 \item $\sigma$ preserves $E_0\subseteq X_0^2$ setwise.
 \end{enumerate}
 
\begin{proof}  (i)  Let $s:=x+y\in S_0$.
 Since $z/E_1=k(s)$, $s+z\in T\subseteq W_0^N$.  As $\sigma\in Aut(N)$, this implies $\sigma(s+z)\in W_0^N$.  By the uniqueness given by Claim 1, $\sigma(z)/E_1=k(\sigma(s))$, but
 $k(\sigma(s))=k^*(\{\sigma(x),\sigma(y)\}$).
 
 (ii)  Suppose $z,z'\in X_1$ and $E_1(z,z')$.  Choose $s\in S_0$ such that $z/E_1=k(s)$ and let $s'=\sigma(s)$.  By (i) twice we have  $\sigma(z)/E_1=k(s')=\sigma(z')/E_1$, hence
 $E_1(\sigma(z),\sigma(z'))$.
 
 (iii)  $Q\subseteq W_1^N\cap X_1$ is obvious.  For the converse, since $Q\subseteq X_1$ and $X_1$ is linearly independent, $Span(Q)\cap X_1=Q$.  As $W_1^N=Span(Q)$, we are finished.
As  both $X_1$ and $W_1^N$ are invariant under any $\sigma\in Aut(N)$, so is $Q$.

(iv)  Unpacking the definitions, note that for $\{x,y\}\in [X_0]^2$, $E_0(x,y)$ holds if and only if $k(x+y)\cap Q\neq\emptyset$ if and only if $k(x+y)\subseteq Q$.
Using this and (iii), for $\{x,y\}\in [X_0]^2$, 
$$E_0(x,y)\quad\Leftrightarrow\quad k(x+y)\subseteq Q\quad\Leftrightarrow\quad k(\sigma(x)+\sigma(y))\subseteq Q\quad\Leftrightarrow \quad E_0(\sigma(x),\sigma(y))$$
\qed
\end{proof}

\end{proof}
 
% With Proposition~\ref{lift} in hand, we can now prove Theorem~\ref{tagBC} by directly coding graphs using one additional unary predicate for a subgroup.
% 
%% \begin{Theorem} \label{tagBC} Tagged abelian groups are Borel complete.  In fact, tagged $\FF_2$-vector spaces with predicates interpreted as $\FF_2$-subspaces is Borel complete.
%% \end{Theorem}
%% 
%% \bp  
%
%\medskip

\noindent {\bf Proof of Theorem~\ref{tagBC}}

\medskip

Fix any tagged abelian group $N$ in the language $L=\{+,0\}\cup\{U_n:n\in\omega\}$ satisfying the conditions of Proposition~\ref{lift}.  
 Let $L_+:=L\cup\{U_+\}$.  As noted in Remark~\ref{shuffle}, any expansion $(N,U_+)$ of $N$ by interpreting $U_+$ as a subgroup of $N$ can be construed as a tagged abelian group.
   Fix, for once and for all, a bijection: $\omega\leftrightarrow X_0/E_0$.  
   We define a Borel embedding from (irreflexive, symmetric) graphs $G=(X_0/E_0,R^G)$ with universe $X_0/E_0$ into
 expansions $N(G)=(N,U_+(G))$ of $N$ with $U_+(G)$ interpreted as a subgroup of $N$.
 Given a graph $G=(X_0/E_0,R^G)$, let
 $$Z_G:=\{z\in X_1:z/E_1=k^*(\{x,y\})\ \hbox{for some $\{x,y\}\in [X_0]^2$ with $G\models R(x/E_0,y/E_0)$}\}$$
 let $U_+(G):=\<Z_G\>$ (a subgroup of $N_0$) and let $N(G):=(N,U_+(G))$.  This mapping $G\mapsto N(G)$ is clearly Borel and $N(G)$ is a tagged abelian group, 
 so we must show it preserves isomorphism in both directions.
 
 First, suppose $h:G\rightarrow H$ is a graph isomorphism.  Then $h\in Sym(X_0/E_0)$, so choose a lifting $\sigma\in Aut_L(N)$.  We argue that $\sigma:N(G)\rightarrow N(H)$
 is an $L_{+}$-isomorphism.  For this, it suffices to show that $\sigma[Z_G]=Z_{H}$ (setwise).  Choose $z\in Z_G$.  Since $z\in X_1$, choose $\{x,y\}\in [X_0]^2$ with $z/E_1=k^*(\{x,y\})$.
 Since $\sigma$ is a lifting of $h$ we have
 $$\sigma(x)/E_0=h(x/E_0)\quad\hbox{and}\quad\sigma(y)/E_0=h(y/E_0)$$
 By Proposition~\ref{lift}(3), $\sigma(z)/E_1=k^*(\{\sigma(x),\sigma(y)\})$.
 Since $z\in Z_G$, $G\models R(x/E_0,y/E_0)$, so as $h$ is a graph isomorphism $H\models R(h(x/E_0),h(y/E_0))$.  
 Combining these, we get that $\{\sigma(x),\sigma(y)\}$ witness that $\sigma(z)\in Z_{H}$.  That $\sigma(z)\in Z_{H}$ implies $z\in Z_G$ is symmetric.
 
Conversely,  suppose $f:N(G)\rightarrow N(H)$ is an $L_+$-isomorphism.  Then clearly $f\in Aut_{L}(N)$.  Additionally, we claim that $f[Z_G]=Z_{H}$ setwise.
To see this, choose $z\in Z_G$.  Since $f\in Aut_{L}(N)$, $f$ permutes $X_1$, hence $f(z)\in X_1$.  But also, since $f[U_+(G)]=U_+^{H}$, $f(z)\in Span(Z_{H})$.
Since $X_1$ is linearly independent, we conclude that $f(z)\in Z_{H}$.  The reverse direction is symmetric.

Since $f\in Aut_{L}(N)$, both $X_0$ and $E_0$ are $f$-invariant, so $f$ induces a permutation $h\in Sym(X_0/E_0)$.  We argue that $h:G\rightarrow H$ is a graph isomorphism.
To see this, choose $\{x,y\}\in [X_0]^2$ with $G\models R(x/E_0,y/E_0)$.
Choose any $z\in X_1$ such that $z/E_1=k^*(\{x,y\})$.  Thus $z\in Z_G$ by definition, so by the note above, $f(z)\in Z_{H}$.  
 By Proposition~\ref{lift}(3), $f(z)/E_1=k^*(\{f(x),f(y)\})$.  Since $k^*$ is a bijection,  $\{f(x),f(y)\}$ is the only witness to $f(z)\in Z_{H}$,
hence $H\models R(f(x)/E_0,f(y)/E_0)$ by the definition of $Z_H$.

\endproof

%\begin{thebibliography}{99}
%
%
%\bibitem{Shc} S. Shelah, {\it Classification Theory\/}, (revised edition)
%North Holland, Amsterdam, 1990.
%
%
%\end{thebibliography}

\end{document}